\documentclass[10pt, a4paper]{amsart}
\usepackage[dvips]{epsfig}
\usepackage{amsmath}
\usepackage{amssymb}
\usepackage{amsfonts}
\usepackage{amsthm}
\usepackage{amsbsy}
\usepackage{amsgen}
\usepackage{amscd}
\usepackage{amsopn}
\usepackage{amstext}
\usepackage{amsxtra}
\usepackage{xypic}

\newtheorem{theorem}{Theorem}[section]

\newtheorem{lemma}[theorem]{Lemma}

\theoremstyle{definition}

\numberwithin{equation}{section}

\begin{document}

\title[Note on a theorem of Bousfield and Friedlander]
{Note on a theorem of Bousfield and Friedlander}
\author{Alexandru E. Stanculescu}
\address{Department of Mathematics and Statistics, McGill University, 805 Sherbrooke Str. West, Montr\'eal,
Qu\'ebec, Canada, H3A 2K6} \email{stanculescu@math.mcgill.ca}


\begin{abstract}
We examine the proof of a classical localization theorem of
Bousfield and Friedlander and we remove the assumption that the
underlying model category be right proper. The key to the argument
is a lemma about factoring in morphisms in the arrow category of a
model category.
\end{abstract}

\maketitle
\section{Introduction}
Let $\mathcal{C}$ be a (Quillen) model category. A
(\emph{left})\emph{Bousfield localization of} $\mathcal{C}$ is
another model category structure on $\mathcal{C}$ having the same
class of cofibrations as the given one and a bigger class of weak
equivalences. There are several methods for constructing left
Bousfield localizations for (some classes of model categories)
$\mathcal{C}$, see e.g. \cite{GJ} and the references therein.

In their work on the construction of the stable homotopy category,
Bousfield and Friendlander introduced (\cite{BF}, Thm. A.7) a method
of localization involving an endofunctor $Q:\mathcal{C}\rightarrow
\mathcal{C}$ with good enough properties. Later on, Bousfield
(\cite{Bou}, Thm. 9.3 and Remark 9.5) improved the result by
weakening the hypotheses on $\mathcal{C}$ and refining the axioms
that $Q$ has to satisfy.

The purpose of this note is to further remove one of the hypotheses
of the Bousfield's version of the original Bousfield and Friedlander
theorem. The details are as follows. Let $\mathcal{C}$ be a model
category together with a functor $Q:\mathcal{C}\rightarrow
\mathcal{C}$. We say that a map $f$ of $\mathcal{C}$ is a
$Q$-\emph{equivalence} if $Q(f)$ is a weak equivalence, and we say
that a map is a $Q$-\emph{fibration} if it has the right lifting
property with respect to all the cofibrations of $\mathcal{C}$ which
are $Q$-equivalences. An object $X$ of $\mathcal{C}$ is
$Q$-\emph{fibrant} if the map $X\rightarrow 1$ is a $Q$-fibration.
Here $1$ denotes the terminal object of $\mathcal{C}$. We prove

\begin{theorem} Let $\mathcal{C}$ be a model category and let
$\gamma:\mathcal{C}\rightarrow Ho(\mathcal{C})$ be the localization
functor. Suppose that there are a functor $Q:\mathcal{C}\rightarrow
\mathcal{C}$ and a natural transformation $\alpha:Id\Rightarrow Q$
satisfying the following properties:

(A1) the functor $Q$ preserves weak equivalences;

(A2) for each $X\in \mathcal{C}$, the map $Q(\alpha_{X})$ is a weak
equivalence and the map $\gamma(\alpha_{Q(X)})$ is a monomorphism.

(A3) $Q$-equivalences are stable under pullbacks along fibrations
between fibrant objects $f:X\rightarrow Y$ such that $\alpha_{X}$
and $\alpha_{Y}$ are weak equivalences.

Then $\mathcal{C}$ admits a left Bousfield localization with the
class of $Q$-equivalences as weak equivalences.
\end{theorem}
The theorem differs from (\cite{Bou}, Thm. 9.3) to the amount that
it doesn't require $\mathcal{C}$ to be right proper. (The resulting
model structure will be right proper because of (A3).) Its proof is
a modification of the proofs given in (\cite{GJ}, Thm. X.4.1) and
(\cite{Bou}, Thm. 9.3). It will be given is section 2 after few
lemmas.
\\

\emph{Note.} The published version of this paper \cite{St} contains
a small mistake: the proof of lemma 2.1($ii$) is wrong. We give here
a correct proof.

\section{Proof of theorem 1.1}
The setting in which we shall work for the next lemmas is the
following. $\mathcal{C}$ is a model category with localization
functor $\gamma:\mathcal{C}\rightarrow Ho(\mathcal{C})$. We are
given a functor $Q:\mathcal{C}\rightarrow \mathcal{C}$ and a natural
transformation $\alpha:Id\Rightarrow Q$ satisfying the following
properties:

(A1) the functor $Q$ preserves weak equivalences;

(A2) for each $X\in \mathcal{C}$, the map $Q(\alpha_{X})$ is a weak
equivalence and the map $\gamma(\alpha_{Q(X)})$ is a monomorphism.

\begin{lemma} Let $\mathcal{K}:=\{ X\in \mathcal{C}\ |\ \alpha_{X}$
$is \ an \ isomorphism \ in \ Ho(\mathcal{C})\}$. We view
$\mathcal{K}$ as a full subcategory of $Ho(\mathcal{C})$. Then

(i) $Q(X)\in \mathcal{K}$ for all $X\in \mathcal{C}$;

(ii) $1\in \mathcal{K}$;

(iii) $\mathcal{K}$ is replete in $Ho(\mathcal{C})$;

(iv) the maps $\gamma(Q(\alpha_{X}))$ and $\gamma(\alpha_{Q(X)})$
are equal.
\end{lemma}
\begin{proof} ($i$) and ($iii$) are clear. For ($ii$), notice that $\alpha_{1}$
is a retract of $\alpha_{Q(1)}$
\[
   \xymatrix{
1 \ar[r]^{\alpha_{1}} \ar[d]_{\alpha_{1}} & Q(1)
\ar[d]_{\alpha_{Q(1)}}
\ar[r] & 1 \ar[d]_{\alpha_{1}}\\
Q(1) \ar[r]_{Q(\alpha_{1})} & Q(Q(1)) \ar[r] & Q(1)\\
}
   \]
and use ($iv$). We now prove ($iv$). By general theory there are:
($a$) a functor $\hat{Q}:Ho(\mathcal{C})\rightarrow Ho(\mathcal{C})$
such that $\hat{Q}\gamma=\gamma Q$, and ($b$) a natural
transformation $\hat{\alpha}:Id\Rightarrow \hat{Q}$ such that
$\hat{\alpha}\gamma=\gamma \alpha$. Let $X$ be an object of
$\mathcal{C}$. We have a commutative diagram
\[
   \xymatrix{
\gamma X \ar[r]^{\gamma \alpha_{X}} \ar[d]_{\gamma \alpha_{X}} &
\gamma Q(X) \ar[d]^{\gamma Q(\alpha_{X})}\\
\gamma Q(X) \ar[r]_{\gamma(\alpha_{Q(X)})} & \gamma Q(Q(X)).\\
}
   \]
Let $g:=\gamma(\alpha_{Q(X)})$, $f:=\gamma Q(\alpha_{X})$ and
$u:=f^{-1}g$. Then $u\hat{\alpha}_{\gamma X}=\hat{\alpha}_{\gamma
X}$, hence $\hat{Q}(u)\hat{Q}(\hat{\alpha}_{\gamma
X})=\hat{Q}(\hat{\alpha}_{\gamma X})$, which implies that
$\hat{Q}(u)$ is the identity map. The commutative diagram
\[
   \xymatrix{
\hat{Q}(\gamma X) \ar[r]^{\hat{\alpha}_{\hat{Q}\gamma X}}
 \ar[d]_{u} &
\hat{Q}^{2}(\gamma X) \ar[d]^{\hat{Q}(u)}\\
\hat{Q}(\gamma X) \ar[r]_{\hat{\alpha}_{\hat{Q}\gamma X}} & \hat{Q}^{2}(\gamma X)\\
}
   \]
implies then that $u$ is the identity, and therefore the maps
$\gamma(Q(\alpha_{X}))$ and $\gamma(\alpha_{Q(X)})$ are equal.
\end{proof}
\begin{lemma} A map of $\mathcal{C}$ is a trivial fibration
iff it is a $Q$-fibration and a $Q$-equivalence.
\end{lemma}
\begin{proof} This is (\cite{GJ}, Lemma X.4.3).
\end{proof}
\begin{lemma}
Let
\[
   \xymatrix{
A \ar[rr] \ar[dr]^{i} \ar[dd] & &
X \ar[dr]^{f} \ar[dd]|-{u}\\
&  B \ar[rr] \ar[dd] & & Y \ar[dd]|-{v}\\
A' \ar[rr] \ar[dr]^{i'} &  &
X' \ar[dr]\\
&  B' \ar[rr] & & Y'\\
 }
   \]
be a (commutative) cube diagram in a model category $\mathcal{E}$.
Suppose that $i$ is a cofibration, $f$ is a fibration between
fibrant objects and $i'$, $u$ and $v$ are weak equivalences. Then
the top face of the cube has a diagonal filler.
\end{lemma}
\begin{proof}
Consider the diagram
\[
   \xymatrix{
A' \ar[r] \ar[d]_{i'} & X' \ar[r]^{u'} \ar[d]
 & \widehat{X'} \ar[d]^{q}\\
B' \ar[r] & Y' \ar[r]^{v'} & \widehat{Y'}\\
}
   \]
where $u'$ and $v'$ are trivial cofibrations and $q$ is a fibration
between fibrant objects. We factor the composite map $B'\rightarrow
\widehat{Y'}$ as a trivial cofibration $B'\rightarrow \widehat{B'}$
followed by a fibration $\widehat{B'}\rightarrow \widehat{Y'}$ and
then take the pullback $P$ of the diagram
\[
   \xymatrix{
& \widehat{X'} \ar[d]^{q}\\
\widehat{B'} \ar[r] & \widehat{Y'}.\\
}
   \]
We factor the canonical map $A'\rightarrow P$ as a trivial
cofibration $A'\rightarrow \widehat{A'}$ followed by a fibration
$\widehat{A'}\rightarrow P$ and we obtain a commutative cube
\[
   \xymatrix{
A' \ar[rr] \ar[dr]^{i} \ar[dd] & &
X' \ar[dr] \ar[dd]|-{u'}\\
&  B' \ar[rr] \ar[dd] & & Y' \ar[dd]|-{v'}\\
\widehat{A'} \ar[rr] \ar[dr]^{\hat{i'}} &  &
 \widehat{X'} \ar[dr]\\
&  \widehat{B'} \ar[rr] & & \widehat{Y'}\\
 }
   \]
in which the maps $\widehat{A'}\rightarrow \widehat{X'}$ and
$\widehat{B'}\rightarrow \widehat{Y'}$ are fibrations between
fibrant objects and the map $\widehat{i'}$ is a weak equivalence.
Composing the above cubes and then taking the pullbacks of the front
and back new faces results in a commutative diagram
\[
   \xymatrix{
A \ar[r] \ar[dr]_{i} \ar[dd] & \widehat{A'}\times_{\widehat{X'}}X
\ar[r] \ar[ddl] \ar[dr]^{p} &
X \ar[dr]^{f} \ar[dd]\\
&  B \ar[r] \ar[dd] & \widehat{B'}\times_{\widehat{Y'}}Y \ar[r] \ar[ddl] & Y \ar[dd]\\
\widehat{A'} \ar[rr] \ar[dr]_{i'} &  &
\widehat{X'} \ar[dr]\\
&  \widehat{B'} \ar[rr] & & \widehat{Y'}.\\
 }
   \]
It follows that the map $p$ is a weak equivalence. As such, $p$ has
a factorisation $qj$, where $j$ is a trivial cofibration and $q$ is
a trivial fibration. Since $i$ was a cofibration and $f$ a
fibration, the the top face of the original cube diagram  has a
diagonal filler.
\end{proof}
\begin{lemma} A cofibration of $\mathcal{C}$ is a $Q$-equivalence iff it has the left
lifting property with respect to every fibration between fibrant
objects which belong to $\mathcal{K}$.
\end{lemma}
\begin{proof}
($\Rightarrow$) Let
\[
   \xymatrix{
A \ar[r] \ar[d]_{i} & X \ar[d]^{f}\\
B \ar[r] & Y\\
}
   \]
be a commutative diagram with $i$ a cofibration $Q$-equivalence and
$f$ a fibration between fibrant objects which belong to
$\mathcal{K}$. Apply the previous lemma to the cube diagram
\[
   \xymatrix{
A \ar[rr] \ar[dr]^{i} \ar[dd] & &
X \ar[dr]^{f} \ar[dd]\\
&  B \ar[rr] \ar[dd] & & Y \ar[dd]\\
Q(A) \ar[rr] \ar[dr]^{Q(i)} &  &
 Q(X) \ar[dr]\\
&  Q(B) \ar[rr] & & Q(Y).\\
 }
 \]
($\Leftarrow$) Let $i:A\rightarrow B$ be a cofibration of
$\mathcal{C}$ which has the left lifting property with respect to
every fibration between fibrant objects which belong to
$\mathcal{K}$. Consider the diagram
\[
   \xymatrix{
A \ar[r]^{\alpha_{A}} \ar[d]_{i} & Q(A) \ar[r]^{u} \ar[d]^{Q(i)}
 & \widehat{Q(A)} \ar[d]^{\widehat{Q(i)}}\\
B \ar[r]^{\alpha_{B}} & Q(Y) \ar[r]^{v} & \widehat{Q(B)}\\
}
   \]
where $u$ and $v$ are trivial cofibrations and $\widehat{Q(i)}$ is a
fibration between fibrant objects. By hypothesis the outer diagram
has a diagonal filler $d$. Applying $Q$ to the previous diagram we
obtain a diagram
\[
   \xymatrix{
A \ar[rr]^{Q(u\alpha_{A})} \ar[d]_{Q(i)} & & Q(\widehat{Q(A)}) \ar[d]^{Q(\widehat{Q(i)})}\\
B \ar[rr]_{Q(v\alpha_{B})} \ar[urr]^{Q(d)} & & Q(\widehat{Q(B)})\\
}
   \]
in which both horizontal arrows are weak equivalences. By the two
out of six property of weak equivalences it follows that $Q(d)$ is a
weak equivalence, hence $i$ is a $Q$-equivalence.
\end{proof}
\begin{lemma}
(i) An object $X$ of $\mathcal{C}$ is $Q$-fibrant iff $X$ is fibrant
and $X\in \mathcal{K}$.

(ii) A map between $Q$-fibrant objects is a $Q$-fibration iff it is
a fibration.
\end{lemma}
\begin{proof}
$(i)$ If $X$ is fibrant and $\alpha_{X}$ is a weak equivalence then
by 2.1 and 2.4 we conclude that $X$ is $Q$-fibrant. Conversely, let
$X$ be $Q$-fibrant. We factor the map $\alpha_{X}$ as $pi$, where
$i:X\rightarrow D$ is a cofibration and $p:D\rightarrow Q(X)$ is a
trivial fibration. Then $i$ is a $Q$-equivalence, so the diagram
\[
   \xymatrix{
  X \ar[r]^{id_{X}} \ar[d]_{i} & X \ar[d]\\
 D \ar[r] & 1\\
}
   \]
has a diagonal filler. Consequently, $\alpha_{X}$ is a retract of
$\alpha_{D}$. But $D\in \mathcal{K}$ by 2.1. Part $(ii)$ follows
from $(i)$ and 2.4.
\end{proof}
\emph{Proof of Theorem 1.1.} Since we have lemma 2.2 it only remains
to show that every arrow $f:X\rightarrow Y$ of $\mathcal{C}$ can be
factored into a cofibration $Q$-equivalence followed by a
$Q$-fibration. The proof follows exactly the proof of (\cite{Bou},
Thm. 9.3) with the difference that we appeal to lemma 2.5. To make
things clear we repeat the argument. Consider the diagram
\[
   \xymatrix{
X \ar[r]^{\alpha_{X}} \ar[d]_{f} & Q(X) \ar[r]^{u} \ar[d]^{Q(f)}
 & \widehat{Q(X)} \ar[d]^{\widehat{Q(f)}}\\
Y \ar[r]^{\alpha_{Y}} & Q(Y) \ar[r]^{v} & \widehat{Q(Y)}\\
}
   \]
where $u$ and $v$ are trivial cofibrations and $\widehat{Q(f)}$ is a
fibration between fibrant objects. The map $\widehat{Q(f)}$ is a
$Q$-fibration by lemma 2.5($ii$). We pull it back along the
$Q$-equivalence $v\alpha_{Y}$ to obtain a $Q$-fibration
$g:E\rightarrow Y$ such that the map $E\rightarrow \widehat{Q(X)}$
is a $Q$-equivalence by (A3). Therefore the canonical map
$X\rightarrow E$ is a $Q$-equivalence. We factor it into a
cofibration $j$ followed by a trivial fibration $p$, and then
$f=(gp)j$ is the desired factorization of $f$.
\\

\textbf{Remark 2.6.} If $\mathcal{C}$ is a combinatorial model
category and $Q$ is an accessible functor, then it follows from
Smith's theorem (\cite{Bek}, Thm. 1.7) that the conclusion of
theorem 1.1 remains valid without imposing the axiom (A3).
\\

\textbf{Acknowledgements.} The result of this paper was obtained
during the author's stay at the CRM Barcelona. We would like to
thank CRM for support and warm hospitality.

\end{document}